\documentclass[12pt]{amsart}
\usepackage{amscd,amsmath,amsthm,amssymb,graphics}
\usepackage{slashbox}
\usepackage{pgf,tikz}
\usetikzlibrary{arrows}
\usepackage{lineno}

\def\NZQ{\Bbb}               

\def\ZZ{{\NZQ Z}}

\def\frk{\frak}               

\def\Phi{{\frk n}}
\def\Phi{{\frk N}}
\def\opn#1#2{\def#1{\operatorname{#2}}} 
\opn\chara{char} \opn\length{\ell} \opn\pd{pd} \opn\rk{rk}
\opn\projdim{proj\,dim} \opn\injdim{inj\,dim} \opn\rank{rank}
\opn\depth{depth} \opn\grade{grade} \opn\height{height}
\opn\embdim{emb\,dim} \opn\codim{codim}

\opn\Tr{Tr} \opn\bigrank{big\,rank}
\opn\superheight{superheight}\opn\lcm{lcm}
\opn\trdeg{tr\,deg}
\opn\reg{reg} \opn\lreg{lreg} \opn\ini{in} \opn\lpd{lpd}
\opn\size{size}\opn\bigsize{bigsize}
\opn\cosize{cosize}\opn\bigcosize{bigcosize}
\opn\sdepth{sdepth}\opn\sreg{sreg}
\opn\link{link}\opn\fdepth{fdepth}
\opn\index{index}
\opn\index{index}
\opn\indeg{indeg}
\opn\N{N}
\opn\SSC{SSC}
\opn\SC{SC}
\opn\lk{lk}
\opn\div{div} \opn\Div{Div} \opn\cl{cl} \opn\Cl{Cl}
\opn\Spec{Spec} \opn\Supp{Supp} \opn\supp{supp} \opn\Sing{Sing}
\opn\Ass{Ass} \opn\Min{Min}\opn\Mon{Mon} \opn\dstab{dstab} \opn\astab{astab}
\opn\Syz{Syz}
\opn\reg{reg}
\opn\Ann{Ann} \opn\Rad{Rad} \opn\Soc{Soc}
\opn\Im{Im} \opn\Ker{Ker} \opn\Coker{Coker} \opn\Am{Am}
\opn\Hom{Hom} \opn\Tor{Tor} \opn\Ext{Ext} \opn\End{End}\opn\Der{Der}
\opn\Aut{Aut} \opn\id{id}

\opn\nat{nat}
\opn\pff{pf}
\opn\Pf{Pf} \opn\GL{GL} \opn\SL{SL} \opn\mod{mod} \opn\ord{ord}
\opn\Gin{Gin} \opn\Hilb{Hilb}\opn\sort{sort}
\opn\initial{init}
\opn\ende{end}
\opn\height{height}
\opn\type{type}
\opn\aff{aff} \opn\con{conv} \opn\relint{relint} \opn\st{st}
\opn\lk{lk} \opn\cn{cn} \opn\core{core} \opn\vol{vol}
\opn\link{link} \opn\star{star}\opn\lex{lex}
\opn\gr{gr}

\def\pot#1#2{#1[\kern-0.28ex[#2]\kern-0.28ex]}

\opn\dirlim{\underrightarrow{\lim}}
\opn\inivlim{\underleftarrow{\lim}}
\def\Implies{\ifmmode\Longrightarrow \else
        \unskip${}\Longrightarrow{}$\ignorespaces\fi}
\def\implies{\ifmmode\Rightarrow \else
        \unskip${}\Rightarrow{}$\ignorespaces\fi}
\def\iff{\ifmmode\Longleftrightarrow \else
        \unskip${}\Longleftrightarrow{}$\ignorespaces\fi}

\let\:=\colon
\newtheorem{Theorem}{Theorem}[section]
 \newtheorem{Lemma}[Theorem]{Lemma}
 \newtheorem{Corollary}[Theorem]{Corollary}
 \newtheorem{Proposition}[Theorem]{Proposition}
 \newtheorem{Remark}[Theorem]{Remark}

 \newtheorem{Definition}[Theorem]{Definition}

\let\epsilon\varepsilon
\let\kappa=\varkappa
\def\qed{\ifhmode\textqed\fi
      \ifmmode\ifinner\quad\qedsymbol\else\dispqed\fi\fi}
\def\textqed{\unskip\nobreak\penalty50
       \hskip2em\hbox{}\nobreak\hfil\qedsymbol
       \parfillskip=0pt \finalhyphendemerits=0}
\def\dispqed{\rlap{\qquad\qedsymbol}}

\opn\dis{dis}
\def\pnt{{\raise0.5mm\hbox{\large\bf.}}}

\opn\Lex{Lex}

\begin{document}

\title{On the asymptotic linearity of
 reduction number}
\author{ Dancheng Lu}

\address{Dancheng Lu, Department of Mathematics, Soochow University, 215006 Suzhou, P.R.China}
\email{ludancheng@suda.edu.cn}



\begin{abstract} Let $R$ be a standard graded  algebra over an infinite  field $K$ and $M$  a finitely generated $\ZZ$-graded $R$-module. Then for any graded ideal $I\subseteq R_+$ of $R$, we show that there exist  integers $\epsilon_1\geq \epsilon_2$ such that  $r(I^nM)=\rho_I(M)n+\epsilon_1$ and $D(I^nM)=\rho_I(M)n+\epsilon_2$ for $n\gg0$. Here $r(M)$ and  $D(M)$ denote the reduction number of $M$ and the maximal degree   of minimal generators of $M$ respectively, and $\rho_I(M)$ is an integer determined by both $M$ and $I$.
\end{abstract}

\maketitle

\section{Introduction}

Unless otherwise stated, we always assume that $R=\bigoplus_{n\geq 0}R_n$ is a standard graded Noetherian algebra over an infinite field $K$, where ``standard graded" means that $R_0=K$ and $R=K[R_1]$. As usual, a nonzero element in $R_1$ is called a {\it linear form} of $R$. Let $M$ be a  finitely generated nonzero $\ZZ$-graded $R$-module.

\begin{Definition}\label{dreduction}{\em A graded ideal $J$ of $R$ is called an {\it $M$-reduction}  if $J$ is generated by linear forms  such that $(JM)_n=M_n$ for $n\gg0$; An $M$-reduction  is called {\it minimal} if  it does not contain any other $M$-reduction. The {\it reduction number} of $M$ with respect to $J$ is defined to be $$r_J(M):=\max\{n\in \ZZ\:\; (JM)_n\neq M_n\}.$$ The {\it reduction number} of $M$ is $$r(M):=\min\{r_J(M)\:\; J \mbox{ is a minimal } M\mbox{-reduction}\}.$$
}\end{Definition}

  Let $I$ be a graded ideal of $R$. In this paper, we are interested in the following natural
problem : is $r(I^nM)$ a linear function of $n$ for all $n\gg  0$?   This problem
is inspired by  the asymptotic behaviour of the so-called Castelnuovo-Mumford
regularity $\mathrm{reg}(I^nM)$. It was first shown in \cite{CHT} and \cite{K}   for the case $R$ being a
polynomial ring over a field, and then in \cite{TW} for the general case (namely, when $R$ is a standard graded algebra over a Noetherian ring with unity) that $\mathrm{reg}(I^nM)$ is a
linear function of $n$ for all $n\gg 0$. Since the reduction number $r(I^nM)$  is less than
or equal to the Castelnuovo-Mumford regularity $\mathrm{reg}(I^nM)$ by \cite[Proposition 3.2]{T2}, it is bounded above  by a
linear function of $n$.

One of the main obstacles to tackle this question lies in the fact that the reduction number is not a homological invariant.
  Hence we can not detect any relations among  the reduction numbers of modules $M_i$ from the short exact  sequence $0\rightarrow M_1\rightarrow M_2\rightarrow M_3\rightarrow 0$. However we find that if both  $M$ and $N$ share the same dimension, then $r(N)\leq r(M)$ provided that $N$ is a quotient module of $M$. It turns out that this simple fact plays an important role.

 To state our main result, we introduce some more notation.   Again let $M$ be a finitely generated nonzero $\ZZ$-graded $R$-module.  We then use $D(M)$ and $d(M)$ to denote  the largest and least degrees of a minimal system of generators of $M$ respectively. In other words:
$$D(M):=\max\{n\in \ZZ\:\; (M/R_+M)_n\neq 0\} \mbox{ and } d(M):=\min\{n\in \ZZ\:\; M_n\neq 0\}.$$

   Recall that a graded  ideal $J$ contained in $I$ is  an  {\it $M$-reduction of $I$} if $JI^nM=I^{n+1}M$ for some $n>0$. Note that here we do not require that $J$ is generated by linear forms, hence this concept is  different from the notion of $M$-reduction given in Definition~\ref{dreduction}, even though one can show that $JR_+^nM=R_+^{n+1}M$ for $n\gg0$ if and only if $(JM)_n=M_n$ for $n\gg 0$.  The integer $\rho_I(M)$ is defined to be: $$\rho_I(M):=\min\{D(J)\:\; J \mbox{ is an } M\mbox{-reduction of }I\}.$$

   We answer our question in positivity by showing:
\begin{Theorem}[Main Theorem]  There exist  integers $\epsilon_1\geq \epsilon_2\geq  d(M)$ such that $r(I^nM)=\rho_I(M)n+\epsilon_1$ and $D(I^nM)=\rho_I(M)n+\epsilon_2$ for $n\gg0$.
\label{main}

  \end{Theorem}

     Combining this result with the main result in \cite{TW}   we see that  $\mathrm{reg}(I^nM), r(I^nM)$ and $D(I^nM)$ are all linear functions of $n$ with the same slope.

 \vspace{2mm}
     There is a local version of reduction number. Let $I\subseteq \mathrm{m}$ be an ideal of a Noetherian local ring $(R,\mathrm{m})$. An ideal $J\subseteq I$ is called a {\em reduction} of $I$ if $JI^n=I^{n+1}$ for some $n>0$. A reduction $J$ of $I$ is called a minimal reduction if it does not contain properly  a reduction of $I$. If $J$ is a minimal reduction of $I$, the {\it reduction number} $I$ with respect to $J$, denoted by $r_J(I)$,  is the least positive integer $n$ such that $JI^n=I^{n+1}$. The reduction number of $I$ is defined to be the integer $r(I):=\min\{r_J(I)\:\; J \mbox{ is a minimal reduction of } I\}$. In \cite{H}, it was proved:

     \vspace{2mm}
      {\bf Hoa's Theorem:} Let $I\subseteq \mathrm{m}$ be an ideal of a local ring $(R,\mathrm{m})$. Then there is an integer $s$ such that for any $n\gg 0$ and any minimal reduction $J$ of $I^n$, one has $r_J(I^n)=s$. In particular, $r(I^n)=s$ for all $n\gg 0$.

 \vspace{2mm}
         Comparing Hoa's Theorem with the result obtained in our paper, we see that the graded and local notions of reduction number  are  very different, especially in their asymptotic behaviour. Namely, one has the following statement:

 \vspace{2mm}
     In the local case, $r(I^n)$ is constant for $n\gg 0$; but in the graded case, $r(I^n)$ is a linear function with a nonzero slope for $n\gg 0$.

\vspace{2mm}
In the last section, we introduce the notion of a {\it generalized regularity function} (see Definition~\ref{function}) for a standard graded algebra over a Noetherian ring with unity, which is a generalization of a {\it regularity function} given in \cite{HR}. We prove:

  \begin{Theorem}
    \label{1.3} Assume that  $R$ is a standard graded algebra over a Noetherian ring with unity.  Let   $I$ be a graded ideal of $R$ and $M$  a finitely generated $\ZZ$-graded $R$-module, and let $\Gamma$ be a generalized regularity function for $R$. Then there exists an integer $e\geq d(M)$ such that $\Gamma(I^nM)=\rho_I(M)n+e$ for $n\gg 0$
\end{Theorem}

\section{Asymptotic linearity}

In this section we will keep the assumptions and notation in  the preceding section.
 Recall that a   linear form $y_1\in R_1$ is {\it filter regular} on $M$ if $0:_My$ is a module of finite length. A sequence $y_1,\ldots,y_r$ with $y_i\in R_1$ is a {\it  filter regular sequence} on $M$ if $y_i$ is filter regular on $M/(y_1,\ldots,y_{i-1})M$ for all $i=1,\ldots,r$.
 \vspace{2mm}

Let $\dim (M)$ denote the Krull dimension of $M$.  We collect some basic properties of a reduction number in the following two lemmas.

\begin{Lemma}\label{Basic} Let $J$ be an $M$-reduction. Then:

{\em (a)}  $r_J(M)\geq D(M)$;

{\em (b)} $r_J(M)=\min\{n\geq D(M)\:\; (JM)_{n+1}=M_{n+1}\}$;

{\em (c)} $r_J(M(-p))=r_J(M)+p$ for all $p\in \ZZ$;

{\em (d)} if $\dim (M)>0$ and  $n\geq r_J(M)$, then $r_J(M_{\geq n})=n$;

{\em (e)} if $\dim (M)>0$, then $r_J(M)=\min\{n\:\; r_J(M_{\geq n})=n\}$.

\end{Lemma}

\begin{proof} (a) Since $JM\subseteq R_+M$, we have  $(R_+M)_n=M_n$ for all $n\geq r_J(M)+1$. Hence $D(M)\leq r_J(M)$.

\vspace{2mm}
(b) Set $r=\min\{n\geq D(M)\:\; (JM)_{n+1}=M_{n+1}\}$. Since $r\geq D(M)$, we have  $D(JM)\leq r+1$, and so $(JM)_{m+1}=M_{m+1}$ for any $m\geq r$. This implies $r_J(M)\leq r$.  The inequality $r\leq r_J(M)$ follows from  $r_J(M)\in\{n\geq D(M)\:\; (JM)_{n+1}=M_{n+1}\}$.

\vspace{2mm}

(c) $r_J(M(-p))=\max\{n\in \ZZ\:\; (JM(-p))_n\neq M(-p)_n\}=$ \\  $\max\{n\in \ZZ\:\; (JM)_{n-p}\neq M_{n-p}\}=\max\{n+p\in \ZZ\:\; (JM)_{n}\neq M_{n}\}=r_J(M)+p$.
\vspace{2mm}

 For the proof of (d) and (e), we first notice that for any $k\in \ZZ$, $(JM)_k=J_1M_{k-1}+J_1R_1M_{k-2}+\cdots=J_1M_{k-1}$. In particular, $(JM)_{k+i}=(JM_{\geq k})_{k+i}$ for all $i\geq 1$.

\vspace{2mm}
(d) Since $n\geq D(M)$, the module $M_{\geq n}$ is generated in degree $n$. In particular $M_n=(M_{\geq n})_n\neq 0$ (for if $M_n=0$ then $M_{\geq n}=0$, so $\dim (M)=0$, a contradiction). This implies $r_J(M_{\geq n})\geq n$ since $(JM_{\geq n})_n=0$. For all $i\geq 1$, we have $(JM_{\geq n})_{n+i}=(JM)_{n+i}=M_{n+i}=(M_{\geq n})_{n+i}$. Hence $r_J(M_{\geq n})\leq n$.
\vspace{2mm}

(e) Set $r=\min\{n\:\; r_J(M_{\geq n})=n\}$. Then $r\leq r_J(M)$ by (d). Since $r_J(M_{\geq r})=r$, we have $(JM)_{r+i}=(JM_{\geq r})_{r+i}=(M_{\geq r})_{r+i}=M_{r+i}$ for all $i\geq 1$, so $r_J(M)\leq r$.  \end{proof}

\begin{Lemma} Assume that $\dim (M)=d$ and let $J$ be  an $M$-reduction. Then

\label{Basic2}

{\em  (a)} $J$ is generated by at least $d$ linear forms;

 {\em (b)} if  $y_1,\ldots,y_d$ is a filter regular sequence on $M$, then $Q=(y_1,\ldots, y_d)$ is a minimal $M$-reduction. In particular, for a generic sequence of  linear forms $y_1,\ldots, y_d$, the ideal $Q=(y_1,\ldots, y_d)$ is a minimal $M$-reduction.

\end{Lemma}

\begin{proof}

(a) Note that $\dim (M/JM)=0$,   one then uses e.g. \cite[Proposition A.4]{BH}.
\vspace{1mm}

(b) From the exact sequence $0\rightarrow 0_M:y_1\rightarrow M(-1)\rightarrow M\rightarrow M/y_1M\rightarrow 0$,   we see that $\dim (M/y_1M)=\dim (M)-1$, by comparing the degrees of Hilbert polynomials of $M$ and $M/y_1M$. Hence $\dim (M/QM)=0$.  It follows that $Q$ is an $M$-reduction and it is minimal by (a). The last statement follows from \cite[Lemma 4.3.1]{HH}, which says a generic linear form is  filter regular on $M$.
\end{proof}

Next, we will show that
$r(M)$ is the reduction number of $M$ with respect to any generic minimal reduction, along a similar line  as given in \cite{T}. For this, we introduce some notation and some basic facts. Let $n\in \ZZ$. We use  $t_n$ for   the dimension of the $K$-space $M_n$ and let $T_i$ be  a  $K$-basis of $M_n$.  Let $x_1,\ldots,x_m$ be a $K$-basis of $R_1$. Then for any ideal $J$ of $R$ generated by $d$ linear forms: $y_1,\ldots, y_d$, there exists a matrix $\alpha=(\alpha_{i,j})\in K^{d\times m}$,
such that $$y_i=\sum_{j=1}^m \alpha_{i,j}x_j$$ for $i=1,\ldots,d$.
We call $\alpha$ the {\it parameterized matrix} of $J$.

 For $n\in \ZZ$, the vector space $(JM)_n=J_1M_{n-1}$ is spanned by vectors $y_ig$, with $g\in T_{n-1}$ and $1\leq i\leq d$. Let $\mathbb{M}_n(\alpha)$ denote the matrix of coefficients of those elements written as linear combinations of elements in $T_n$. Then $J$ is a minimal $M$-reduction if and only if $d=\dim (M)$ and there exists  $n\geq D(M)$ such that $\mbox{rank}\  \mathbb{M}_{n+1}(\alpha)=t_{n+1}$ in view of Lemma~\ref{Basic}(b). In this case: $$r_J(M)=\min\{n\geq D(M)\:\; \mbox{rank}\  \mathbb{M}_{n+1}(\alpha)=t_{n+1}\}.$$

\begin{Proposition} Suppose that $\dim (M)=d$. Then for a generic sequence of  linear forms $y_1,\ldots,y_d$ of $R$, we have $r(M)=r_{(y_1,\ldots,y_d)}(M).$
\label{generic}
\end{Proposition}

\begin{proof} Step 1: Let $U=(u_{i,j})_{d\times m}$ be a matrix of indeterminates, and set  $$R_U=R\otimes_K K(U) \mbox{\ and\ } M_U=M\otimes_K K(U).$$ Here $K(U)$ is the fractional field of  the  polynomial ring $K[U]=K[u_{i,j}\:\; 1\leq i\leq d, 1\leq j\leq m]$. Then $R_U$ is a standard graded algebra over $K(U)$ and $M_U$ is a finitely generated $\ZZ$-graded $R_U$-module.  Note that the $K$-basis $T_n$ of $M_n$ is also a $K_U$-basis of $(M_U)_n$ for all $n\in \ZZ$, and that the $K$-basis  $x_1,\ldots,x_m$ of $R_1$ is also a $K(U)$-basis of $(R_U)_1$.

 Set $z_i=\sum_{j=1}^m u_{ij}x_j\in (R_U)_1$ for $i=1,\ldots,d$. We claim that $z_1,\ldots,z_d$ is a filter regular sequence on $M_U$. In fact, let $P$ be an associated prime ideal of $M_U$ with $P\nsupseteq (R_U)_+$. Since $P=pR_U$ for some associated prime ideal  $p$ of $M$, we see that if $z_1\in P$, then $x_i\in p$ for $i=1, \ldots, m$, and so $P\supseteq (R_U)_+$, a contradiction. Hence $z_1$ is filter regular on $M_U$. By induction, $z_1,\ldots,z_d$ is a filter regular sequence on $M_U$, as claimed. Therefore $(z_1,\ldots,z_d)$ is a minimal $M_U$-reduction by Lemma~\ref{Basic2}(b).
 \vspace{1mm}

 Step 2: For $n\in \ZZ$, we define the matrix $\mathbb{M}_n(U)$ similarly  as $\mathbb{M}_n(\alpha)$, that is, $\mathbb{M}_n(U)$ is the matrix of coefficients of  elements $z_ig$ with $1\leq i\leq d$ and $g\in T_{n-1}$ written as linear combinations of elements in $T_n$. It is not hard to see $\mathbb{M}_n(\alpha)$ is a specialization of $\mathbb{M}_n(U)$.

 \vspace{2mm}

 Step 3: We show that $$r(M)=r_{(z_1,\ldots,z_d)}M_U=\min\{n\geq D(M_U)(=D(M))\:\; \mbox{rank}\ \mathbb{M}_{n+1}(U)=t_{n+1} \}.$$

 Put $r=r_{(z_1,\ldots,z_d)}M_U$. Then $\mbox{rank}\  \mathbb{M}_{r+1}(U)=t_{r+1}$. Since $K$ is an infinite field, there exists a matrix $\alpha\in K^{d\times m}$ such that $\mbox{rank}\ \mathbb{M}_n(\alpha)=\mbox{rank}\ \mathbb{M}_n(U)$. This implies $r(M)\leq r$. On the other hand, it is clear that $\mbox{rank}\ \mathbb{M}_n(U)\geq \mbox{rank}\ \mathbb{M}_n(\alpha)$ for any $\alpha\in K^{d\times m}$. Hence $r(M)=r$. The second equality follows from the discussion before this proposition.
  \vspace{2mm}

 Step 4: Let $f(U)$ be a nonzero $t_{r+1}$-minor of $\mathbb{M}_{r+1}(U)$. Then any $\alpha\in K^{d\times m}$ with $f(\alpha)\neq 0$ corresponds to a minimal reduction $J$ such that $r_J(M)=r(M)$.
\end{proof}
\vspace{2mm}

The following corollary is crucial to the proof of our main result.
\vspace{2mm}

\begin{Corollary}\label{important} Let $M_i,i=1,\ldots,n$ be finitely generated nonzero $\ZZ$-graded $R$-modules with the same dimension $d$.

 {\em (a)} If $M_1\twoheadrightarrow M_2$ is an epimorphism of graded modules, then $r(M_1)\geq r(M_2)$;

 {\em (b)}  $r(\oplus_{i=1}^nM_i)=\max\{r(M_i)\:\; 1\leq i\leq n \}$;

{\em (c)} If $0\rightarrow M_1\rightarrow M_2\rightarrow M_3\rightarrow 0$ is a short exact sequence of graded modules, then:
        $r(M_3)\leq r(M_2)\leq \max\{r(M_1),r(M_3)\}.$
\label{1.4}
\end{Corollary}
\begin{proof}  (a) In view of Proposition~\ref{generic}, there exists a  sequence of linear forms $y_1,\ldots,y_d$ such that the ideal $Q:=(y_1,\ldots,y_d)$ is  an $M_i$-reduction and $r(M_i)=r_Q(M_i)$ for $i=1,2$. Since $r_Q(M_i)=\max\{n\in \ZZ\:\; (M_i/QM_i)_n\neq 0\}$, the desired inequality follows from the epimorphism: $$M_1/QM_1\rightarrow M_2/QM_2\rightarrow 0,$$ which is obtained by tensoring the epimorphism $M_1\twoheadrightarrow M_2$ with $R/Q$.

The proofs of (b) and (c) are similar.
\end{proof}

  It is known that  $\dim (M/I^nM)$ keeps constant as $n$ grows since $\sqrt{\mathrm{Ann}(M/I^nM)}(=\sqrt{I^n+\mathrm{Ann}(M)})$ is independent of $n$. But this is not the case for $\dim (I^nM)$. For instance, let $R=K[x,y]$, $M=R/(x^2)$ and $I=(x)$. Then $\dim (I^2M)=0$ but $\dim (IM)=1$. However we have:

\begin{Lemma} The function  $\dim (I^nM)$ is constant for $n\gg0$
\label{krullconstant}
\end{Lemma}
\begin{proof} This fact follows from  $\dim (I^nM)=\dim (R/\mathrm{Ann}(I^nM))$ and the sequence of ideals $\mathrm{Ann}(I^nM)$ is increasing and hence stationary for $n\gg 0$.
\end{proof}

 We record \cite[Lemma 3.1]{T} in the following lemma  for the later use. Note that in this result  we only require that $R_0$ is a Noetherian ring with unity.

\begin{Lemma} $D(I^nM)\geq \rho_I(M)n+d(M)$ for all $n>0$.
\label{D(M)}

\end{Lemma}

We now in the  position to prove the main result of this paper.

\begin{Theorem} There exist  integers $\epsilon_1\geq \epsilon_2\geq  d(M)$ such that $r(I^nM)=\rho_I(M)n+\epsilon_1$ and $D(I^nM)=\rho_I(M)n+\epsilon_2$ for $n\gg0$.
\label{main}
\end{Theorem}

\begin{proof}

Let $\mathcal{U}:=\bigoplus_{n\geq 0} I^nM/R_+(I^nM)$ be the module over the Rees ring $\mathfrak{R}(I)$ of $I$. Then $\mathfrak{R}(I)$ has a natural bigraded algebra structure and $\mathcal{U}$ is a bigraded $\mathfrak{R}(I)$-module. It follows that $\mathrm{reg}( \mathcal{U}_n)$ is a linear function for $n\gg0$ by \cite[Theorem 2.2]{TW}. Since $\mathcal{U}_n=I^nM/R_+(I^nM)$ has the finite length, it follows that
$D(I^nM)=\mathrm{reg}(\mathcal{U}_n)$ and so it is a linear function for $n\gg 0$ with  $$D(I^nM)\geq \rho_I(M)n+d(M),\ \forall n>0.$$   On the other hand,  we know that $D(I^nM)\leq \mathrm{reg} (I^nM)$ and for $n\gg0$, one has $$\mathrm{reg} (I^nM)=\rho_I(M)n+\epsilon,$$ where $\epsilon$ is a positive integer independent of $n$, see \cite{TW}. From these facts, it immediately follows that $$D(I^nM)=\rho_I(M)n+\epsilon_2$$ with $\epsilon_2\geq d(M)$ for $n\gg 0$.
\vspace{2mm}

Next, we consider the function $r(I^nM)$. To this end, we set $$Q_n:=r(I^nM)-\rho_I(M)n.$$ Then $Q_n\geq d(M)$ by Lemma~\ref{D(M)} together with  Lemma~\ref{Basic}(a).  Let $J$ be an $M$-reduction of $I$ such that $D(J)=\rho_I(M)$, and let $u_1,\ldots, u_t$ be a minimal generating system of $J$.  Then for $n\gg 0$ we have the following epimorphism of graded $R$-modules: $$\rightarrow\bigoplus_{j=1}^tI^{n-1}M(-p_j)\rightarrow I^nM\rightarrow 0,$$ where $p_i=\mathrm{deg}\ u_i$ for $i=1,\ldots, t$.
In fact, if $n$ is large  enough, then $I^nM=JI^{n-1}M=u_1I^{n-1}M+\cdots+u_tI^{n-1}M$. Thus the desired map can be obtained by compounding the homogeneous surjective maps: $I^{n-1}M(-p_j)\rightarrow u_jI^{n-1}M$, which is defined by: $a\mapsto u_ja$.

   From the epimorphism defined  above, it follows that $$r(I^nM)\leq \max\{r(I^{n-1}M)+p_i\:\; 1\leq i\leq t\}$$ by Lemmas~\ref{krullconstant} and \ref{Basic}(c) and Corollary~\ref{important}. Since  $\rho_I(M)=\max\{p_i\:\; 1\leq i\leq t\}$, it follows that $$Q_n\leq Q_{n-1}$$ for $n\gg 0$  and so $Q_n$ is eventually constant with a value $\epsilon_1\geq d(M)$. This proves that $r(I^nM)=\rho_I(M)n+\epsilon_1$ for $n\gg 0$. Finally, the inequality $\epsilon_1\geq \epsilon_2$ follows from Lemma~\ref{Basic}(a).
\end{proof}

We conclude this section  with the following question:

\vspace{2mm}

\noindent {\bf Question:} {\em What do we can say about the function $r(M/I^nM)$?}

\vspace{2mm}

 This function is non-decreasing from the beginning on by the epimorphism  $$M/I^{n+1}M\rightarrow M/I^nM\rightarrow 0.$$
Since $r(M/I^nM)\leq \mathrm{reg} (M/I^nM)\leq \max\{\mathrm{reg}(M), \mathrm{reg} (I^nM)+1\}$, it is also bounded above by a linear function. In some particular cases, for example when $M=R$ and $I=R_+$, it is  a linear function asymptotically. However we do not know if it is indeed  asymptotically linear in general.

 \section{Generalized Regularity Functions}

In the last section we will extend the epimorphism appeared in  the proof of Theorem~\ref{main} to a long exact sequence and use it to prove the asymptotic linearity of {\it generalized regularity function.}  First we change our setting. From now on,  let $R=\oplus_{n\geq 0} R_n$ be  a standard graded Noetherian algebra over $R_0=A$, where $A$ is  a Noetherian ring with unity. We denote by $\mathcal{M}_R$ the category of finitely generated graded $R-$modules.

\begin{Definition}\label{function} {\em  A {\it generalized regularity function} for $R$ is a function $\Gamma$ which assigns each $M\in \mathcal{M}_R$ an integer $\Gamma(M)$ such that for all  $M,N,P\in \mathcal{M}_R$, one has:

(1) if $M\cong N$, then $\Gamma(M)=\Gamma(N)$;

(2) $\Gamma(M(-p))=\Gamma(M)+p$ for all $p\in \ZZ$;

(3) $\Gamma(M)\geq D(M)$;

(4) $\Gamma(M\oplus N)=\max\{\Gamma(M),\Gamma(N)\}$,

(5) if $0\rightarrow M\rightarrow N\rightarrow P\rightarrow 0$ is a short exact sequence of graded modules, then $\Gamma(P)\leq \max\{\Gamma(N),\Gamma(M)-1\}$.}
\end{Definition}

This concept generalizes the notion  of regularity function defined in \cite{HR} in two hands. Firstly we do not require that $R_0$ is field; Secondly, some requirements in the definition of regularity function (see \cite[Definition 1.1]{HR}) are dropped. Note that  the condition that  $\Gamma(M)\geq D(M)$ in our definition does not appear in \cite[Definition 1.1]{HR}, but it can be deduced from that  in view of \cite[Proposition 1.2(a)]{HR}.   We begin with:
\begin{Lemma} \label{repeat} Let  $0\rightarrow M_t\rightarrow M_{t-1}\rightarrow \cdots\rightarrow M_1\rightarrow M_0\rightarrow 0$ be an exact sequence of graded modules. Then $\Gamma(M_0)\leq \max\{\Gamma(M_1),\Gamma(M_2)-1,\ldots,\Gamma(M_t)-t+1\}.$
\end{Lemma}
\begin{proof} Split the exact sequence into the $(t-1)$ short exact sequences and apply the  fifth condition of Definition~\ref{function} to those  sequences.
\end{proof}

In the following lemma we will present an useful exact sequence of graded modules, which is an improvement of the exact sequence appeared in the proof of \cite[Theorem 4.7.6]{BH}. Let $[t]$ denote the set of integers $\{1,\ldots,t\}$.

\begin{Lemma}
\label{sequence}
Let $J$ be an $M$-reduction of $I$ and assume that $J$ is minimally generated by $u_1,\ldots,u_t$ with $\deg u_i=p_i$ for $i=1,\ldots,t$. Then for all $n\gg0$, we have the following exact sequence of graded modules:
\[
0\rightarrow I^{n-t}M(-\sum_{i=1}^tp_i)\rightarrow \cdots \rightarrow  \bigoplus_{T\subseteq [t], |T|=i}I^{n-i}M(-\sum_{j\in T}p_j)  \rightarrow \cdots \rightarrow I^nM\rightarrow 0.
\]

\end{Lemma}
\begin{proof} Let $K.=K.(u_1,\ldots,u_t;M)$ be the Koszul complex and $\partial$ be its differential. Since $\partial(K_{i+1})\subseteq IK_i$ for each $i$,   we have the following  subcomplex of $K.$ for all integer $n\geq t$:
$$K.^{(n)}: \  0\rightarrow I^{n-t}K_t\rightarrow I^{n-t+1}K_{t-1}\rightarrow \cdots\rightarrow I^{n-1}K_1\rightarrow I^nK_0\rightarrow  0.$$

We will show that $K.^{(n)}$ is exact for $n\gg0$ (and thus our result follows.)

Fix $m>0$ such that $JI^mM=I^{m+1}M$. Then $I^{n-i}K_{i}=J^{n-i-m}I^mK_{i}$ and by Artin-Rees Lemma it follows that $$Z_i(K.^{(n)})=I^{n-i}K_{i}\cap Z_i(K.)=J(J^{n-i-m-1}I^mK_{i}\cap Z_i(K.))$$ for $i=0,\ldots,t$ and for $n\gg 0$.

Now given $n$ large enough such that the equalities above hold simultaneously. Let $a\in Z_i(K.^{(n)})$. Then $a=\sum _{j=1}^t u_ja_j$ with $a_j\in J^{n-i-m-1}I^mK_{t-i}\cap Z_i(K.)$.
Let $e_1,\ldots,e_t$ be a $K$- basis of $K_1(u_1,\ldots,u_t; R)$ with $\partial_R (e_j)=u_j$ for $j=1,\ldots,t$. Here $\partial_R$ denotes the differential of $K_1(u_1,\ldots,u_t; R)$. Then $w=\sum_{j=1}^t e_j a_j\in I^{n-i-1}K_{i+1}$ and $\partial (w)=a-\sum_{j=1}^t e_j \partial(a_i)=a$ by \cite[Proposition 1.6.2]{BH}. Thus $K.^{(n)}$ is indeed exact.
\end{proof}

\begin{Remark}
\label{remark}
{\em
 Artin-Rees Lemma says that if $N$ is a submodule of $M$ then there exists an integer $m>0$ such that $I^nM\cap N=I^{n-m}(I^mM\cap N)$ for all $n>m$. However the assumption that $N$ is a submodule of $M$ is not essential. In fact, if  $N$ is not a submodule of $M$ then $I^nM\cap N=I^nM\cap (N\cap M)=I^{n-m}(I^mM\cap (N\cap M))=I^{n-m}(I^mM\cap N)$ for all $n>m$. Therefore Artin-Rees Lemma still holds. We use Artin-Rees Lemma of this version in the proof of Lemma~\ref{sequence}.
}

\end{Remark}

  In the proof of our last result we will use the same strategy  as in the proof of \cite[Theorem 5]{K}.

\begin{Theorem} Let   $I$ be a graded ideal of $R$ and $M\in \mathcal{M}_R$, and let $\Gamma$ be a generalized regularity function for $R$. Then there exists an integer $e\geq d(M)$ such that $\Gamma(I^nM)=\rho_I(M)n+e$ for $n\gg 0$

\end{Theorem}

\begin{proof} Put $P=\rho_I(M)$ and let $Q_n=\Gamma(I^nM)-nP$. Then $Q_n\geq d(M)$ by Lemma~\ref{D(M)} and the third condition in Definition~\ref{function}.  Let $J$ be an $M$-reduction of $I$ such that $D(J)=\rho_I(M)$.  Apply Lemma~\ref{repeat} to the exact sequence in Lemma~\ref{sequence} (see Remark~\ref{remark}(3)), we have:\  \  $Pn+Q_n$ $$\leq \max\{(n-1)P+Q_{n-1}+P,(n-2)P+Q_{n-2}+2P-1,\ldots, (n-t)P+Q_{n-t}+tP-t+1\}.$$
It follows that $Q_n\leq \max\{Q_{n-1}, Q_{n-2}-1,\ldots, Q_{n-t}-t+1\}$ for $n\gg 0$. Put $T_n=\max\{Q_{n-1}, Q_{n-2},\ldots, Q_{n-t}\}$. Then $T_{n+1}\leq T_n$ for $n\gg 0$ and so $T_n$ is eventually constant with a value $e\geq d(M)$.  Let $m$ be an integer such that $T_n=e$ for all $n\geq m$. We claim that $Q_n=e$ for all $n\geq m$.

Assume on the contrary that $Q_m\neq e$. Since $Q_n\leq T_n$,  we have  $Q_m <e$. It follows that $Q_{m+1}\leq \max\{Q_m, Q_{m-1}-1,\ldots, Q_{m-t+1}-t+1\}\leq \max\{Q_m, T_m-1\}<e$. Thus $Q_n<e$ for all $n\geq m$ by induction, a contradiction. Hence $Q_m=e$ and  in the same reason we have $Q_n=e$ for all $n\geq m$, as claimed. This implies $\Gamma(I^nM)=Pn+e$ for all $n\gg 0$.  \end{proof}

\vspace{2mm}
\noindent {\bf Acknowledgement}

 We thank the anonymous referee for helpful comments. In particular, we are grateful for presenting us a much simpler proof for Lemma~\ref{krullconstant} and a more  inspiring  proof for the asymptotic linearity of $D(I^nM)$.

\end{document}